\documentclass[runningheads,a4paper]{llncs}
\usepackage{amssymb}
\usepackage{graphicx}
\usepackage{url}

\usepackage[table]{xcolor}
\usepackage{xspace}
\usepackage{booktabs}
\usepackage{pdflscape}
\usepackage{tikz}
\usetikzlibrary{calc,decorations.markings}
\newcommand\polymake{\texttt{polymake}\xspace}
\newcommand\topcom{\texttt{TOPCOM}\xspace}
\newcommand\mptopcom{\texttt{MPTOPCOM}\xspace}
\newcommand\mts{\texttt{mts}\xspace}
\newcommand\MPI{\texttt{MPI}\xspace}
\newcommand{\gkz}{{\mathop{gkz}}}

\newcommand{\conv}{\mathop{conv}}
\newcommand{\HST}{\mathop{HST}}
\newcommand{\cyclic}{\mathcal C}
\newcommand{\RR}{\mathbb R}
\begin{document}

\mainmatter

\title{New counts for the number of triangulations\\ of cyclic polytopes} 
\author{Michael Joswig\inst{1} \and  Lars Kastner\inst{1}}
\authorrunning{Joswig and Kastner}
\institute{
Institut f{\"u}r Mathematik,\\
 TU Berlin,\\
 Str.\ des 17. Juni 136, 10623 Berlin, Germany\\
 \email{joswig@math.tu-berlin.de},\\
 \url{http://page.math.tu-berlin.de/~joswig/}\\
 \email{kastner@math.tu-berlin.de},\\
 \url{http://page.math.tu-berlin.de/~kastner/}
}
\maketitle

\begin{abstract}
  We report on enumerating the triangulations of cyclic polytopes with the new software \mptopcom.
  This is relevant for its connection with higher Stasheff--Tamari orders, which occur in category theory and algebraic combinatorics.
\end{abstract}



\section{Introduction}

For an integer $d\geq 1$ the $d$-th \emph{moment curve} is the map
\[
  \mu_d:\RR\to\RR^d \,,\ t \mapsto (t,\ldots,t^{d}) \enspace .
\]
Picking $n$ real numbers $t_1<t_2<\dots<t_n$, where $n>d$, the convex hull
\begin{equation}\label{eq:cyclic}
  \cyclic(n,d) \ = \ \conv\{\mu_d(t_1),\mu_d(t_2),\dots,\mu_d(t_n)\}
\end{equation}
is the $d$-dimensional \emph{cyclic polytope} with $n$ vertices.
The combinatorics of $\cyclic(n,d)$ is given by \emph{Gale's evenness criterion}; cf. \cite[Theorem~0.7]{Ziegler:1995}.
In particular, the combinatorial type does not depend on the values $t_1,t_2,\dots,t_n$ but just on their number.
The cyclic polytopes are neighborly, and hence their $f$-vectors attain McMullen's upper bound \cite[Theorem~8.23]{Ziegler:1995}.
The \emph{higher Stasheff--Tamari orders} are certain poset structures on the set of all triangulations of $\cyclic(n,d)$.
Their study was initiated by Kapranov and Voevodsky \cite{KapranovVoevodsky:Stasheff-Tamari} in the context of category theory; see also \cite{EdelmanReiner:Stasheff-Tamari}, \cite{Rambau:1997} and \cite{RambauReiner:Stasheff-Tamari}.
Here we address the problem raised in \cite[\S5.2]{KapranovVoevodsky:Stasheff-Tamari}, which asks for determining the number of triangulations of $\cyclic(n,d)$.
We report on new computational results, obtained via the new software \mptopcom \cite{mptopcom}.
This verifies and extends previous results of Rambau and Reiner \cite[Table~1]{RambauReiner:Stasheff-Tamari}, which were obtained with \topcom \cite{topcom}.
The general question remains wide open.
Notice that the planar case $d=2$ gives the Catalan numbers.

Triangulations of polytopes and of finite point configurations are the subject of the monograph \cite{Triangulations} by De Loera, Rambau and Santos.
Cyclic polytopes are discussed in \cite[\S6.1]{Triangulations}.
We are indebted to J\"org Rambau and Francisco Santos for suggesting to apply \mptopcom to cyclic polytopes and for many useful comments on an earlier version of this text.

This research is carried out in the framework of Matheon supported by Einstein Foundation Berlin. 
Further partial support by DFG (SFB-TRR 109 and SFB-TRR 195) is gratefully acknowledged.

\section{The first higher Stasheff--Tamari order}

Let $P\subset\RR^d$ be a finite point configuration.
A \emph{circuit} of $P$ is a minimally affinely dependent subconfiguration.
A \emph{triangulation} of $P$ is a subdivision of the convex hull $\conv(P)$ whose vertices form a subset of the points in $P$.
Two triangulations of $P$ \emph{differ by a flip} if they agree outside a circuit.
Here we are interested in the triangulations of the point configuration given by the vertices of $\cyclic(n,d)$ and their flips.

There is a canonical projection of the $(d{+}1)$-dimensional simplex $\cyclic(d+2, d+1)$ onto $\cyclic(d+2,d)$ by forgetting the last coordinate.
There are precisely two triangulations of $\cyclic(d+2,d)$, and these correspond to projecting the lower and the upper hull of $\cyclic(d+2,d+1)$.
Consequently, we call them the \emph{lower} and the \emph{upper triangulation} of $\cyclic(d+2,d)$, respectively.
From the construction (\ref{eq:cyclic}) it is immediate that each circuit of $\cyclic(n,d)$ looks like $\cyclic(d+2,d)$.
Combined with the observation on the two triangulations of $\cyclic(d+2,d)$, this has far reaching consequences for the structure of the triangulations of $\cyclic(n,d)$ for arbitrary $n>d$.

Let $\Delta$ and $\Delta'$ be two triangulations of $\cyclic(n,d)$ which differ by a flip.
Then there is subset $C$ of the vertices of cardinality $d+2$ such that $\Delta$ and $\Delta'$ restricted to $C$ look like the upper and the lower triangulations of $\cyclic(d+2,d)$.
If $\Delta$ is the lower and $\Delta'$ is the upper triangulation, then we call the flip $[\Delta\rightsquigarrow\Delta']$ from $\Delta$ to $\Delta'$ an \emph{up-flip}.
Conversely, the reverse flip $[\Delta'\rightsquigarrow\Delta]$ is a \emph{down-flip}.
In this case we write
\begin{equation}\label{eq:HST1}
  \Delta \ \leq_1 \ \Delta' \enspace .
\end{equation}
The partial ordering on the set of all triangulations of $\cyclic(n,d)$ which is obtained as the transitive and reflexive closure of the relation (\ref{eq:HST1}) is the \emph{first higher Stasheff-Tamari order}, denoted as $\HST_1(n,d)$; cf.\ \cite[Definition 6.1.18]{Triangulations} and \cite{Rambau:1997}.
Figure~\ref{fig:HST1-62} below shows $\HST_1(6,2)$ as an example.

On the same set of triangulations of $\cyclic(n,d)$ there is a second natural partial ordering, the \emph{second higher Stasheff-Tamari order}, $\HST_2(n,d)$; cf.\ \cite[Definition 6.1.16]{Triangulations}.
It is known that $\HST_1(n,d)$ is a weaker partial order than $\HST_2(n,d)$.
Moreover, these two orders coincide for $d\leq 3$ and $n-d\in\{1,2,3\}$; cf.\ \cite{RambauReiner:Stasheff-Tamari}.
In general, it is open whether or not they agree.

\begin{figure}
\strut\hfill
\begin{minipage}{.4\textwidth}\centering
\begin{tikzpicture}[baseline=2cm]
\draw[step=1, black!20, thin] (-1.5,-.5) grid (2.5, 4.5);
\draw[thick] (-1,1) -- (0,0) -- (1,1) -- (2,4) -- cycle;
\draw[thick] (-1,1) -- (1,1);
\foreach \i in {-1,0,1,2}{
  \fill[black,radius=2pt] (\i,\i*\i) circle;
}
\foreach \i in {1,2,3,4}{
  \node[label=below:{\colorbox{white}{\i}}] at ($ (\i,0) - (2,0) $) {};
}
\end{tikzpicture}\\
123, 134\\
(8,2,8,6)
\end{minipage}
\hfill
\begin{minipage}{.4\textwidth}\centering
\begin{tikzpicture}[baseline=2cm]
\draw[step=1, black!20, thin] (-1.5,-.5) grid (2.5, 4.5);
\draw[thick] (-1,1) -- (0,0) -- (1,1) -- (2,4) -- cycle;
\draw[thick] (0,0) -- (2,4);
\foreach \i in {-1,0,1,2}{
  \fill[black,radius=2pt] (\i,\i*\i) circle;
}
\foreach \i in {1,2,3,4}{
  \node[label=below:{\colorbox{white}{\i}}] at ($ (\i,0) - (2,0) $) {};
}\end{tikzpicture}\\
124, 234\\
(6,8,2,8)
\end{minipage}
\hfill\strut
\caption{Lower (left) and upper (right) triangulations of $C(4,2)$ with their GKZ-vectors.
  In the lower triangulation the gaps are 4 and 2, i.e., even; whereas in the upper triangulation the gaps are 3 and 1.  Here $d=2$ is even.}
\label{fig:flip:42}
\end{figure}
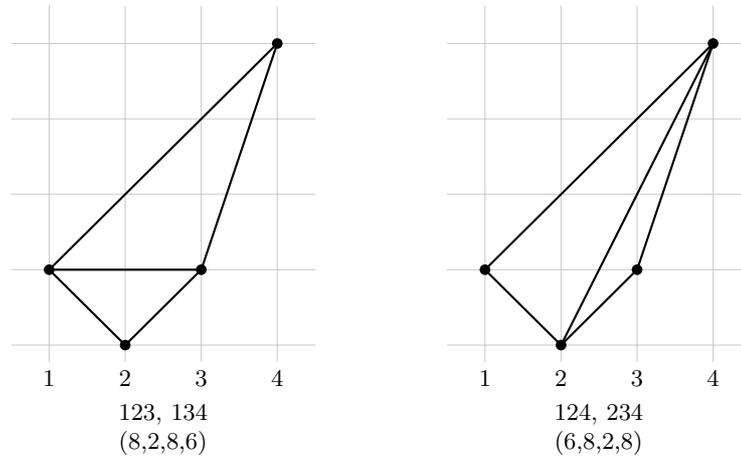

\begin{figure}
\newcommand{\vertexA}{1}
\newcommand{\vertexB}{3}
\newcommand{\vertexC}{2}
\newcommand{\vertexD}{5}
\newcommand{\vertexE}{4}
\strut\hfill
\begin{minipage}{.4\textwidth}\centering
\begin{tikzpicture}[
                    x  = {(1cm,0cm)},
                    y  = {(0cm,1cm)},
                    z  = {(-.4cm,-.15cm)},
                    scale = 1,
                    color = {lightgray}]

  \definecolor{pointcolor_Q1}{rgb}{ 1,0,0 }
  \tikzstyle{pointstyle_Q1} = [fill=pointcolor_Q1]

  \coordinate (v0_Q1) at (-1.75, 3.5, -8.25);
  \coordinate (v1_Q1) at (-0.75, 0.5, -1.25);
  \coordinate (v2_Q1) at (0.25, -0.5, -0.25);
  \coordinate (v3_Q1) at (1.25, 0.5, 0.75);

  \definecolor{edgecolor_Q1}{rgb}{ 0,0,0 }
  \definecolor{facetcolor_Q1}{rgb}{ 0.4667,0.9255,0.6196 }
  \tikzstyle{facestyle_Q1} = [fill=facetcolor_Q1, fill opacity=0.7, draw=edgecolor_Q1, line width=1 pt, line cap=round, line join=round]

  \draw[facestyle_Q1] (v0_Q1) -- (v3_Q1) -- (v2_Q1) -- (v0_Q1) -- cycle;
  \draw[facestyle_Q1] (v2_Q1) -- (v1_Q1) -- (v0_Q1) -- (v2_Q1) -- cycle;
  \draw[facestyle_Q1] (v2_Q1) -- (v3_Q1) -- (v1_Q1) -- (v2_Q1) -- cycle;

  \fill[pointcolor_Q1] (v2_Q1) circle (1 pt);
  \node at (v2_Q1) [text=black, inner sep=0.5pt, below, draw=none, align=left] {\vertexB};

  \draw[facestyle_Q1] (v1_Q1) -- (v3_Q1) -- (v0_Q1) -- (v1_Q1) -- cycle;

  \fill[pointcolor_Q1] (v1_Q1) circle (1 pt);
  \node at (v1_Q1) [text=black, inner sep=0.5pt, left, draw=none, align=left] {\vertexC};
  \fill[pointcolor_Q1] (v3_Q1) circle (1 pt);
  \node at (v3_Q1) [text=black, inner sep=0.5pt, right, draw=none, align=left] {\vertexE};
  \fill[pointcolor_Q1] (v0_Q1) circle (1 pt);
  \node at (v0_Q1) [text=black, inner sep=0.5pt, above, draw=none, align=left] {\vertexA};

  \definecolor{pointcolor_P3}{rgb}{ 1,0,0 }
  \tikzstyle{pointstyle_P3} = [fill=pointcolor_P3]

  \coordinate (v0_P3) at (-1, 1, -1);
  \coordinate (v1_P3) at (-2, 4, -8);
  \coordinate (v2_P3) at (1, 1, 1);
  \coordinate (v3_P3) at (2, 4, 8);

  \definecolor{edgecolor_P3}{rgb}{ 0,0,0 }
  \definecolor{facetcolor_P3}{rgb}{ 0.4667,0.9255,0.6196 }
  \tikzstyle{facestyle_P3} = [fill=facetcolor_P3, fill opacity=0.7, draw=edgecolor_P3, line width=1 pt, line cap=round, line join=round]

  \draw[facestyle_P3] (v0_P3) -- (v3_P3) -- (v1_P3) -- (v0_P3) -- cycle;
  \draw[facestyle_P3] (v2_P3) -- (v0_P3) -- (v1_P3) -- (v2_P3) -- cycle;
  \draw[facestyle_P3] (v2_P3) -- (v3_P3) -- (v0_P3) -- (v2_P3) -- cycle;

  \fill[pointcolor_P3] (v0_P3) circle (1 pt);
  \node at (v0_P3) [text=black, inner sep=0.5pt, below left, draw=none, align=left] {\vertexC};

  \draw[facestyle_P3] (v1_P3) -- (v3_P3) -- (v2_P3) -- (v1_P3) -- cycle;

  \fill[pointcolor_P3] (v1_P3) circle (1 pt);
  \node at (v1_P3) [text=black, inner sep=0.5pt, above, draw=none, align=left] {\vertexA};
  \fill[pointcolor_P3] (v3_P3) circle (1 pt);
  \node at (v3_P3) [text=black, inner sep=0.5pt, left, draw=none, align=left] {\vertexD};
  \fill[pointcolor_P3] (v2_P3) circle (1 pt);
  \node at (v2_P3) [text=black, inner sep=0.5pt, right, draw=none, align=left] {\vertexE};

  \definecolor{pointcolor_Q2}{rgb}{ 1,0,0 }
  \tikzstyle{pointstyle_Q2} = [fill=pointcolor_Q2]

  \coordinate (v0_Q2) at (-1.25, 0.5, -0.75);
  \coordinate (v1_Q2) at (-0.25, -0.5, 0.25);
  \coordinate (v2_Q2) at (0.75, 0.5, 1.25);
  \coordinate (v3_Q2) at (1.75, 3.5, 8.25);

  \definecolor{edgecolor_Q2}{rgb}{ 0,0,0 }

  \definecolor{facetcolor_Q2}{rgb}{ 0.4667,0.9255,0.6196 }

  \tikzstyle{facestyle_Q2} = [fill=facetcolor_Q2, fill opacity=0.7, draw=edgecolor_Q2, line width=1 pt, line cap=round, line join=round]

  \draw[facestyle_Q2] (v2_Q2) -- (v1_Q2) -- (v0_Q2) -- (v2_Q2) -- cycle;
  \draw[facestyle_Q2] (v0_Q2) -- (v3_Q2) -- (v2_Q2) -- (v0_Q2) -- cycle;
  \draw[facestyle_Q2] (v1_Q2) -- (v3_Q2) -- (v0_Q2) -- (v1_Q2) -- cycle;

  \fill[pointcolor_Q2] (v0_Q2) circle (1 pt);
  \node at (v0_Q2) [text=black, inner sep=0.5pt, below left, draw=none, align=left] {\vertexC};

  \draw[facestyle_Q2] (v2_Q2) -- (v3_Q2) -- (v1_Q2) -- (v2_Q2) -- cycle;

  \fill[pointcolor_Q2] (v2_Q2) circle (1 pt);
  \node at (v2_Q2) [text=black, inner sep=0.5pt, right, draw=none, align=left] {\vertexE};
  \fill[pointcolor_Q2] (v3_Q2) circle (1 pt);
  \node at (v3_Q2) [text=black, inner sep=0.5pt, left, draw=none, align=left] {\vertexD};
  \fill[pointcolor_Q2] (v1_Q2) circle (1 pt);
  \node at (v1_Q2) [text=black, inner sep=0.5pt, below, draw=none, align=left] {\vertexB};

   
  \draw[thick,->, black] (-5,-2,-8)node[anchor=south east]{} -- (-4,-2,-8) node[anchor=west] {$x$};
  \draw[thick,->, black] (-5,-2,-8) -- (-5,-1,-8) node[anchor=south] {$y$};
  \draw[thick,->, black] (-5,-2,-8) -- (-5,-2,-7) node[anchor=north east] {$z$};
\end{tikzpicture}\\
  1234, 1245, 2345 \\
  (84,96,24,96,84)
\end{minipage}
\hfill
\begin{minipage}{.4\textwidth}\centering
\begin{tikzpicture}[
                    x  = {(1cm,0cm)},
                    y  = {(0cm,1cm)},
                    z  = {(-.4cm,-.15cm)},
                    scale = 1,
                    color = {lightgray}]

  \definecolor{pointcolor_P2}{rgb}{ 1,0,0 }
  \tikzstyle{pointstyle_P2} = [fill=pointcolor_P2]

  \coordinate (v0_P2) at (-2, 4, -8);
  \coordinate (v1_P2) at (0, 0, 0);
  \coordinate (v2_P2) at (-1, 1, -1);
  \coordinate (v3_P2) at (2, 4, 8);

  \definecolor{edgecolor_P2}{rgb}{ 0,0,0 }
  \definecolor{facetcolor_P2}{rgb}{ 0.4667,0.9255,0.6196 }
  \tikzstyle{facestyle_P2} = [fill=facetcolor_P2, fill opacity=0.7, draw=edgecolor_P2, line width=1 pt, line cap=round, line join=round]

  \draw[facestyle_P2] (v2_P2) -- (v3_P2) -- (v0_P2) -- (v2_P2) -- cycle;
  \draw[facestyle_P2] (v1_P2) -- (v2_P2) -- (v0_P2) -- (v1_P2) -- cycle;
  \draw[facestyle_P2] (v1_P2) -- (v3_P2) -- (v2_P2) -- (v1_P2) -- cycle;

  \fill[pointcolor_P2] (v2_P2) circle (1 pt);
  \node at (v2_P2) [text=black, inner sep=0.5pt, below left, draw=none, align=left] {\vertexC};

  \draw[facestyle_P2] (v0_P2) -- (v3_P2) -- (v1_P2) -- (v0_P2) -- cycle;

  \fill[pointcolor_P2] (v0_P2) circle (1 pt);
  \node at (v0_P2) [text=black, inner sep=0.5pt, above, draw=none, align=left] {\vertexA};
  \fill[pointcolor_P2] (v3_P2) circle (1 pt);
  \node at (v3_P2) [text=black, inner sep=0.5pt, left, draw=none, align=left] {\vertexD};
  \fill[pointcolor_P2] (v1_P2) circle (1 pt);
  \node at (v1_P2) [text=black, inner sep=0.5pt, below, draw=none, align=left] {\vertexB};

  \definecolor{pointcolor_Q1}{rgb}{ 1,0,0 }
  \tikzstyle{pointstyle_Q1} = [fill=pointcolor_Q1]

  \coordinate (v0_Q1) at (-1.5, 4, -8.125);
  \coordinate (v1_Q1) at (0.5, 0, -0.125);
  \coordinate (v2_Q1) at (1.5, 1, 0.875);
  \coordinate (v3_Q1) at (2.5, 4, 7.875);

  \definecolor{edgecolor_Q1}{rgb}{ 0,0,0 }
  \definecolor{facetcolor_Q1}{rgb}{ 0.4667,0.9255,0.6196 }
  \tikzstyle{facestyle_Q1} = [fill=facetcolor_Q1, fill opacity=0.7, draw=edgecolor_Q1, line width=1 pt, line cap=round, line join=round]

  \draw[facestyle_Q1] (v2_Q1) -- (v1_Q1) -- (v0_Q1) -- (v2_Q1) -- cycle;
  \draw[facestyle_Q1] (v1_Q1) -- (v3_Q1) -- (v0_Q1) -- (v1_Q1) -- cycle;
  \draw[facestyle_Q1] (v2_Q1) -- (v3_Q1) -- (v1_Q1) -- (v2_Q1) -- cycle;

  \fill[pointcolor_Q1] (v1_Q1) circle (1 pt);
  \node at (v1_Q1) [text=black, inner sep=0.5pt, below, draw=none, align=left] {\vertexB};

  \draw[facestyle_Q1] (v0_Q1) -- (v3_Q1) -- (v2_Q1) -- (v0_Q1) -- cycle;

  \fill[pointcolor_Q1] (v0_Q1) circle (1 pt);
  \node at (v0_Q1) [text=black, inner sep=0.5pt, above, draw=none, align=left] {\vertexA};
  \fill[pointcolor_Q1] (v3_Q1) circle (1 pt);
  \node at (v3_Q1) [text=black, inner sep=0.5pt, left, draw=none, align=left] {\vertexD};
  \fill[pointcolor_Q1] (v2_Q1) circle (1 pt);
  \node at (v2_Q1) [text=black, inner sep=0.5pt, right, draw=none, align=left] {\vertexE};

   
  \draw[thick,->, black] (-5,-2,-8) node[anchor=south east]{} -- (-4,-2,-8) node[anchor=west] {$x$};
  \draw[thick,->, black] (-5,-2,-8) -- (-5,-1,-8) node[anchor=south] {$y$};
  \draw[thick,->, black] (-5,-2,-8) -- (-5,-2,-7) node[anchor=north east] {$z$};
\end{tikzpicture}
  1235, 1345 \\
  (96,48,96,48,96)
\end{minipage}
\hfill\strut

\caption{Lower (left) and upper (right) triangulations of $C(5,3)$ with their GKZ-vectors.
  In the lower triangulation the gaps are 4 and 2, i.e., even; whereas in the upper triangulation the gaps are 5, 3 and 1.  Here $d=3$ is odd.}
\label{fig:flip:53}
\end{figure}
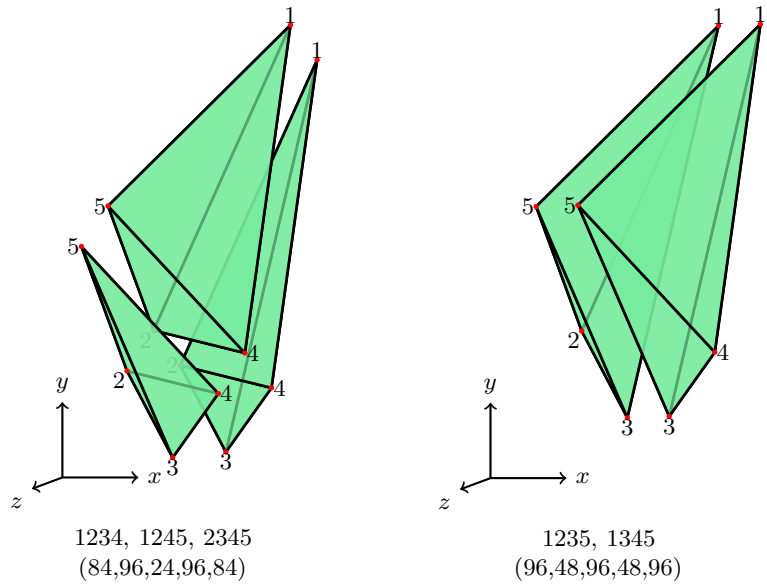

\section{GKZ-vectors}
For a triangulation $\Delta$ of an affine spanning point configuration $P\subset\RR^d$ the \emph{GKZ-vector} is
\[
  \gkz_\Delta \ = \ \bigl( \gkz_\Delta(p) \mid p \in P \bigr) \enspace ,
\]
where $\gkz_\Delta(p)$ is the sum of the normalized volumes of those simplices in $\Delta$ which contain $p$ as a vertex.
The \emph{normalized volume} is the Euclidean volume multiplied by $d!$.

In order to determine the GKZ-vectors of triangulations of cyclic polytopes, we need to choose coordinates.
To keep it simple we can take the lattice points $\mu_d(1),\mu_d(2),\dots,\mu_d(n)$ as the vertices of $\cyclic(n,d)$.
The normalized volume of any $d$-simplex spanned by $\mu_d(i_1),\mu_d(i_2),\dots,\mu_d(i_{d+1})$ with $i_1<i_2<\dots<i_{d+1}$, is the Vandermonde determinant
\begin{equation}\label{eq:vandermonde}
  \det\left(
    \begin{array}{cccc}
      1 & i_1 & \cdots & i_1^d \\
      1 & i_2 & \cdots & i_2^d \\
      \vdots & \vdots & \ddots & \vdots \\
      1 & i_{d+1} & \cdots & i_{d+1}^d
    \end{array}\right)
  \ = \
  \prod_{1\leq k<\ell\leq d+1} (i_\ell-i_k) \enspace .
\end{equation}
In particular, these values do not change when we replace the standard parameters $1,2,\dots,n$ for the moment curve by any other set of $n$ consecutive integers; cf.\ Figure \ref{fig:flip:42}, where we chose the parameters $-1,0,1$ and~$2$, while we keep the labels $1,2,3,4$.
We fix the natural ordering of the vertices on the moment curve in order to identify GKZ-vectors of triangulations of $\cyclic(n,d)$ with vectors in $\RR^n$.
The following basic observation is crucial.

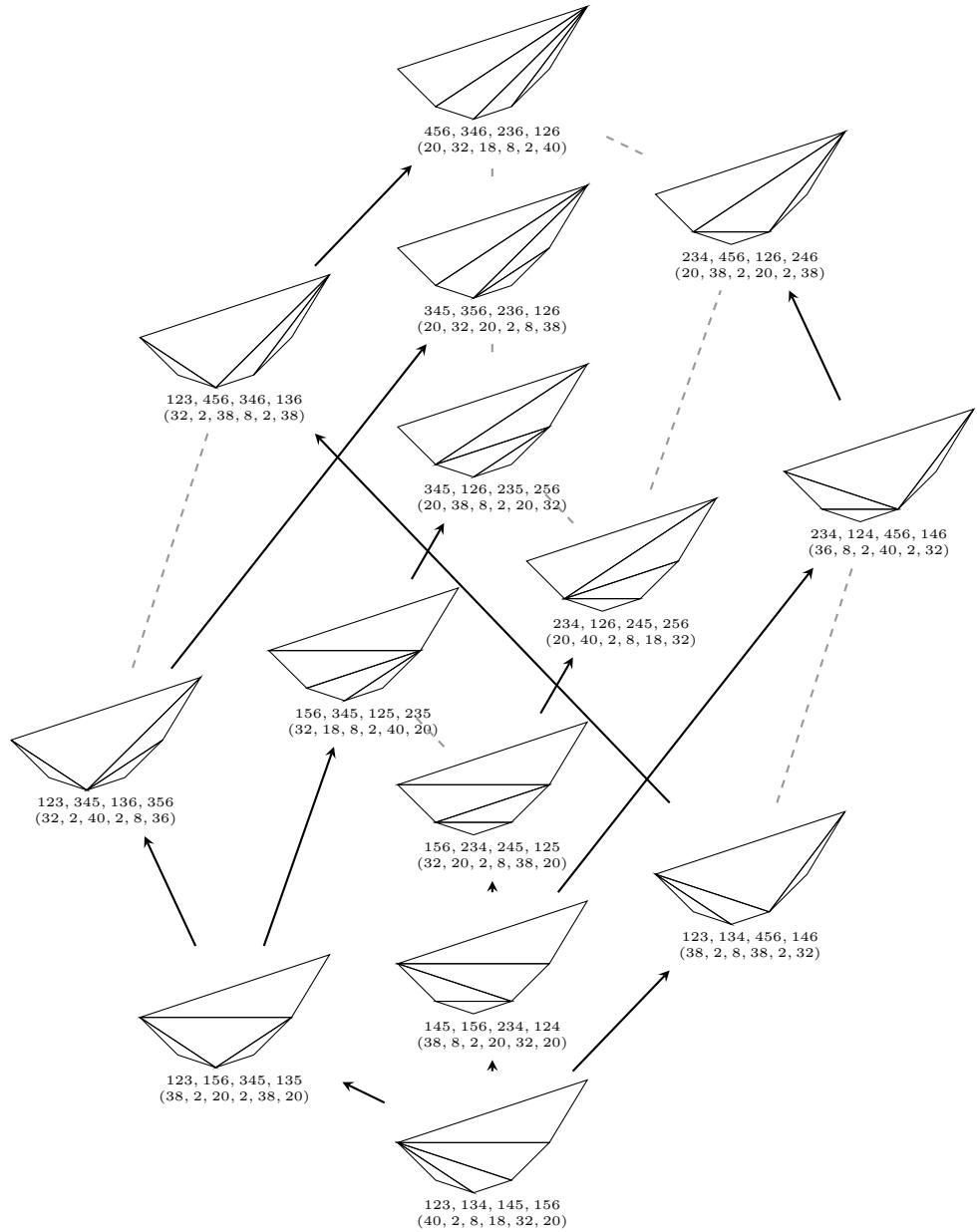
\begin{figure}

\renewcommand{\triangle}[1]{
   \xdef\p{}
   \foreach \n in {#1}{
      \xdef\p{\p (N\n) --}
   }
   \draw \p cycle;
}
\newcommand{\triangulation}[6]{
\begin{array}{c}
\begin{tikzpicture}[scale=.5]
\newcommand{\yscale}{3}
\coordinate (N1) at (-2,4/\yscale);
\coordinate (N2) at (-1,1/\yscale);
\coordinate (N3) at (0,0);
\coordinate (N4) at (1,1/\yscale);
\coordinate (N5) at (2,4/\yscale);
\coordinate (N6) at (3,9/\yscale);
\triangle{#1}
\triangle{#2}
\triangle{#3}
\triangle{#4}
\end{tikzpicture}\\
\foreach \n in {#1}{\n},
\foreach \n in {#2}{\n},
\foreach \n in {#3}{\n},
\foreach \n in {#4}{\n}\\ %
(#6)
\end{array}
}
\begin{tikzpicture}[scale=1.7]
\def\xstretch{1}
\def\ystretch{1.4}
\coordinate (T0coord)  at (\xstretch*0,\ystretch*0);
\coordinate (T1coord)  at (\xstretch*2,\ystretch*1.5);
\coordinate (T2coord)  at (-\xstretch*2,\ystretch*4.5);
\coordinate (T3coord)  at (\xstretch*0,\ystretch*6);
\coordinate (T4coord)  at (-\xstretch*2,\ystretch*0.7);
\coordinate (T5coord)  at (-\xstretch*3,\ystretch*2.25);
\coordinate (T6coord)  at (\xstretch*0,\ystretch*5);
\coordinate (T7coord)  at (-\xstretch*1,\ystretch*2.75);
\coordinate (T8coord)  at (\xstretch*0,\ystretch*4);
\coordinate (T9coord)  at (\xstretch*0,\ystretch*1);
\coordinate (T10coord) at (\xstretch*0,\ystretch*2);
\coordinate (T11coord) at (\xstretch*1,\ystretch*3.25);
\coordinate (T12coord) at (\xstretch*3,\ystretch*3.75);
\coordinate (T13coord) at (\xstretch*2,\ystretch*5.3);

\newcommand{\TrvA}{1}
\newcommand{\TrvB}{2}
\newcommand{\TrvC}{3}
\newcommand{\TrvD}{4}
\newcommand{\TrvE}{5}
\newcommand{\TrvF}{6}

\node[above, font=\tiny] (T0) at (T0coord) {$\triangulation{\TrvA,\TrvB,\TrvC}{\TrvA,\TrvC,\TrvD}{\TrvA,\TrvD,\TrvE}{\TrvA,\TrvE,\TrvF}{T_0}{40,2,8,18,32,20}$};
\node[above, font=\tiny] (T1) at (T1coord) {$\triangulation{\TrvA,\TrvB,\TrvC}{\TrvA,\TrvC,\TrvD}{\TrvD,\TrvE,\TrvF}{\TrvA,\TrvD,\TrvF}{T_1}{38,2,8,38,2,32}$};
\node[above, font=\tiny] (T2) at (T2coord) {$\triangulation{\TrvA,\TrvB,\TrvC}{\TrvD,\TrvE,\TrvF}{\TrvC,\TrvD,\TrvF}{\TrvA,\TrvC,\TrvF}{T_2}{32,2,38,8,2,38}$};
\node[above, font=\tiny] (T3) at (T3coord) {$\triangulation{\TrvD,\TrvE,\TrvF}{\TrvC,\TrvD,\TrvF}{\TrvB,\TrvC,\TrvF}{\TrvA,\TrvB,\TrvF}{T_3}{20,32,18,8,2,40}$};
\node[above, font=\tiny] (T4) at (T4coord) {$\triangulation{\TrvA,\TrvB,\TrvC}{\TrvA,\TrvE,\TrvF}{\TrvC,\TrvD,\TrvE}{\TrvA,\TrvC,\TrvE}{T_4}{38,2,20,2,38,20}$};
\node[above, font=\tiny] (T5) at (T5coord) {$\triangulation{\TrvA,\TrvB,\TrvC}{\TrvC,\TrvD,\TrvE}{\TrvA,\TrvC,\TrvF}{\TrvC,\TrvE,\TrvF}{T_5}{32,2,40,2,8,36}$};
\node[above, font=\tiny] (T6) at (T6coord) {$\triangulation{\TrvC,\TrvD,\TrvE}{\TrvC,\TrvE,\TrvF}{\TrvB,\TrvC,\TrvF}{\TrvA,\TrvB,\TrvF}{T_6}{20,32,20,2,8,38}$};
\node[above, font=\tiny] (T7) at (T7coord) {$\triangulation{\TrvA,\TrvE,\TrvF}{\TrvC,\TrvD,\TrvE}{\TrvA,\TrvB,\TrvE}{\TrvB,\TrvC,\TrvE}{T_7}{32,18,8,2,40,20}$};
\node[above, font=\tiny] (T8) at (T8coord) {$\triangulation{\TrvC,\TrvD,\TrvE}{\TrvA,\TrvB,\TrvF}{\TrvB,\TrvC,\TrvE}{\TrvB,\TrvE,\TrvF}{T_8}{20,38,8,2,20,32}$};
\node[above, font=\tiny] (T9) at (T9coord) {$\triangulation{\TrvA,\TrvD,\TrvE}{\TrvA,\TrvE,\TrvF}{\TrvB,\TrvC,\TrvD}{\TrvA,\TrvB,\TrvD}{T_9}{38,8,2,20,32,20}$};
\node[above, font=\tiny] (T10) at (T10coord) {$\triangulation{\TrvA,\TrvE,\TrvF}{\TrvB,\TrvC,\TrvD}{\TrvB,\TrvD,\TrvE}{\TrvA,\TrvB,\TrvE}{T_{10}}{32,20,2,8,38,20}$};
\node[above, font=\tiny] (T11) at (T11coord) {$\triangulation{\TrvB,\TrvC,\TrvD}{\TrvA,\TrvB,\TrvF}{\TrvB,\TrvD,\TrvE}{\TrvB,\TrvE,\TrvF}{T_{11}}{20,40,2,8,18,32}$};
\node[above, font=\tiny] (T12) at (T12coord) {$\triangulation{\TrvB,\TrvC,\TrvD}{\TrvA,\TrvB,\TrvD}{\TrvD,\TrvE,\TrvF}{\TrvA,\TrvD,\TrvF}{T_{12}}{36,8,2,40,2,32}$};
\node[above, font=\tiny] (T13) at (T13coord) {$\triangulation{\TrvB,\TrvC,\TrvD}{\TrvD,\TrvE,\TrvF}{\TrvA,\TrvB,\TrvF}{\TrvB,\TrvD,\TrvF}{T_{13}}{20,38,2,20,2,38}$};
\tikzstyle tree=[-stealth, thick]
\tikzstyle nontree=[thick, dashed, black!40]
\path[tree] (T0) edge (T1);
\path[tree] (T0) edge (T4);
\path[tree] (T0) edge (T9);
\path[tree] (T1) edge (T2);
\path[tree] (T2) edge (T3);
\path[tree] (T4) edge (T5);
\path[tree] (T4) edge (T7);
\path[tree] (T5) edge (T6);
\path[tree] (T7) edge (T8);
\path[tree] (T9) edge (T10);
\path[tree] (T9) edge (T12);
\path[tree] (T10) edge (T11);
\path[tree] (T12) edge (T13);
\path[nontree] (T6) edge (T3);
\path[nontree] (T13) edge (T3);
\path[nontree] (T2) edge (T5);
\path[nontree] (T6) edge (T8);
\path[nontree] (T8) edge (T11);
\path[nontree] (T11) edge (T13);
\path[nontree] (T7) edge (T10);
\path[nontree] (T1) edge (T12);
\end{tikzpicture}

\caption{First higher Stasheff-Tamari order $HST_1(6,2)$ with reverse search tree marked.
The lowest triangulation has the lexicographically largest GKZ-vector.}
\label{fig:HST1-62}
\end{figure}

\begin{proposition}\label{prop:upflip}
  Let $\Delta$ and $\Delta'$ be two triangulations of $\cyclic(n,d)$ related by a flip $[\Delta\rightsquigarrow\Delta']$.
  Then we have
  \[
    \Delta\leq_1\Delta'\ \iff \
    \left\{
      \begin{array}{ll}
        \gkz_\Delta>_{lex}\gkz_{\Delta'}  & \mbox{ if $d$ even,} \\
        \gkz_\Delta<_{lex}\gkz_{\Delta'}  & \mbox{ if $d$ odd.}
      \end{array}
    \right.
  \]
\end{proposition}
\begin{proof}
  Since each circuit looks like $C(d+2,d)$ it suffices to consider the case
  $n=d+2$. We exploit the relationship of the triangulations of $C(d+2,d)$ with
  the upper and lower hull of $C(d+2,d+1)$ previously explained.

  The Oriented Gale's Evenness Criterion from \cite[Corollary 6.1.9]{Triangulations} describes the upper and lower facets of $C(d+2,d+1)$.
  Let $F\subseteq C(d+2,d+1)$ be a facet, then $F$ can be written as a subset of $[d+2]$,
  the set of indices of vertices in $F$.  The \emph{gaps} of $F$ are the
  elements of $[n]\setminus F$. A gap $i$ of $F$ is \emph{even} if the number
  of elements in $F$ that are larger than $i$ is even. It is called \emph{odd}
  otherwise. Correspondingly, a facet is called odd/even if all its gaps are
  odd/even. The odd facets correspond to the upper triangulation of $C(d+2,d)$
  and the even facets give rise to the lower triangulation of $C(d+2,d)$.
   
  Assume that $1$ is a gap of $F$. Since every facet of $C(d+2,d+1)$ is a
  simplex, $F$ must be $\{2,3,\dots,d+2\}$ and $1$ is the only gap of $F$.
  Hence, if $d$ is odd, then $F$ is even.
  Conversely, if $d$ is even, then $F$ must be odd.
  We conclude that, if $d$ is odd, then all odd facets contain $1$.
  However, if $d$ is even, then only the even facets contain $1$.
  
  Assume now that $d$ is even.
  The odd case is similar.

  Let $\Delta$ and $\Delta'$ be the lower and upper triangulations of $C(d+2,d)$, i.e. $\Delta\leq_1\Delta'$.
  Then $\Delta$ contains all the even facets of $C(d+2,d+1)$.
  But any even facet contains $1$, thus the first entry of $\gkz_{\Delta}$ is the entire normalized volume of $\cyclic(d+2,d)$.
  The facet $\{2,3,\ldots,d+2\}$ is odd, and hence it belongs to $\Delta'$.
  Since it does not contain $1$, we infer that $\gkz_\Delta(1) > \gkz_{\Delta'}(1)$.
  Hence we obtain $\gkz_\Delta>_{lex}\gkz_{\Delta'}$.
  
  This argument can be reversed, and this completes the proof. \qed
\end{proof}

Figure~\ref{fig:flip:42} and Figure~\ref{fig:flip:53}  depicts the situation considered in the proof above for $(n,d)=(4,2)$ and $(n,d)=(5,3)$, respectively.
The interest in Proposition~\ref{prop:upflip} comes from the following.

In \cite{Imai:2002} Imai et al. described an algorithm for computing all (regular) triangulations of a given point configurations, which is based on the reverse search enumeration scheme of Avis and Fukuda \cite{AF93}.
That algorithm, which we call \emph{down-flip reverse search}, was improved and implemented by Skip Jordan with the authors of this extended abstract \cite{mptopcom}.
The basic idea is to orient each flip according to lexicographic ordering of the GKZ-vectors.
Then down-flip reverse search produces a directed spanning tree of those triangulations which can be obtained from some seed triangulation by monotone flipping; cf.\ \cite[\S5.3.2]{Triangulations}.
For the cyclic polytopes we arrive at two choices for orienting the flips, one by GKZ-vectors, one according to the first higher Stasheff--Tamari order.
Now Proposition~\ref{prop:upflip} says that these two choices fortunately agree.

\begin{corollary}
  Down-flip reverse search computes a directed spanning tree of the first Stasheff-Tamari poset $\HST_1(n,d)$, rooted at the triangulation with the lexicographically largest GKZ-vector.
  For $d$ even, the root is the lowest triangulation of $\cyclic(n,d)$, whereas, for $d$ odd, the root is the highest triangulation.
  In particular, each triangulation of a cyclic polytope can be obtained by monotone flipping from the respective roots.
\end{corollary}

The first higher Stasheff--Tamari order $\HST_1(6,2)$ with GKZ-vectors is shown in Figure~\ref{fig:HST1-62}.

\newcommand{\verified}[1]{\textcolor{black}{#1}}
\newcommand{\fixed}[1]{\textcolor{red}{#1}}
\newcommand{\new}[1]{\textcolor{blue}{#1}}
\newcommand{\toobig}[1]{\textcolor{black}{#1}}
\newcommand{\planned}{}

\begin{table}[th]\centering
  \caption{The number of triangulations of $\cyclic(c+d,d)$.
    The column $d=2$ contains the Catalan numbers, while the row $c=4$ is known by results of Azaola and Santos \cite{AzaolaSantos:2002}.
    The rows $c\in\{1,2,3\}$ are trivial and only listed for completeness.
    The row $c=5$ and the column $d=2$ are marked for their relevance to Question~\ref{qst:beta}.
    Our new results are written in blue; the rest of the table agrees with \cite[Table~1]{RambauReiner:Stasheff-Tamari}.}
\newcommand{\quotSec}[1]{\cellcolor{black!10}{#1}}
\newcommand{\cellQuot}[1]{\cellcolor{black!10}{#1}}
\label{tab:summary}
\begin{tabular}{crrrrrrrrrrrrrrr}
\toprule
c $\backslash$ d: & \cellQuot{2} & 3 & 4 & 5 & 6 &  7  & 8\\
\midrule
     1 & \cellQuot{1} & 1 & 1 & 1 & 1 &  1 & 1 \\
     2 & \cellQuot{2} & 2 & 2 & 2 & 2 &  2 & 2 \\
     3 & \quotSec{5} & 6 & 7 & 8 & 9 & 10 & 11 \\
     4 & \quotSec{14} & 25 & 40 & 67 & 102 & 165 & 244 \\
\cellQuot{5} & \quotSec{42} & \quotSec{138} & \quotSec{357} & \quotSec{1\,233} & \quotSec{3\,278} & \quotSec{12\,589} & \quotSec{\verified{35\,789}}\\
     6 & \quotSec{132} & 972 & 4\,824 & \verified{51\,676} & \verified{340\,560} & \verified{6\,429\,428} &\new{68\,007\,706}\\
7 & \quotSec{429} & 8\,477 & 96\,426 & \verified{5\,049\,932} & \verified{132\,943\,239} & \planned\\
8 & \quotSec{1\,430} & 89\,405 & \verified{2\,800\,212} & \new{1\,171\,488\,063}\\
9 & \quotSec{4\,862} & \verified{1\,119\,280} & \verified{116\,447\,760} \\
10 & \quotSec{16\,796} & \verified{16\,384\,508} & \planned\\
11 & \quotSec{\verified{58\,786}} & \verified{276\,961\,252} \\
12 & \quotSec{208\,012} & \new{5\,349\,351\,298}\\
\toprule
c $\backslash$ d: &  9 & 10 & 11 & 12 & 13 & 14\\
\midrule
1 &   1 &  1 &  1 &  1 &  1 & 1 \\
2 &   2 &  2 &  2 &  2 &  2 & 2 \\
3 &  12 & 13 & 14 & 15 & 16 & 17 \\
4 &  387 & 562 & 881 & 1\,264 & 1\,967 & 2\,798\\
\cellQuot{5} &  \quotSec{\verified{159\,613}} & \quotSec{\verified{499\,900}} & \quotSec{\verified{2\,677\,865}} & \quotSec{\verified{9\,421\,400}} & \quotSec{\verified{62\,226\,044}} & \quotSec{\new{247\,567\,074}}\\
\bottomrule
\end{tabular}
\end{table}

\section{Computations with \mptopcom}
The open source software \mptopcom is designed for computing triangulations in a massively parallel setup.
Its algorithm is the down-flip reverse search method of Imai et al. \cite{Imai:2002} with several improvements as described in \cite{mptopcom}.
As its key feature reverse search is output sensitive, and this makes it attractive for extremely large enumeration problems.
Our parallelization, based on the \MPI protocol, employs \emph{budgeting} for load balancing; cf.\ \cite{mts_tutorial,AD17}.
In this way \mptopcom can enumerate the (regular) triangulations of much larger point sets than other software before; extensive experiments are described in \cite[\S7]{mptopcom}.
\mptopcom uses linear algebra and basic data types from \polymake \cite{polymake}, triangulations and flips from \topcom \cite{topcom} and the budgeted parallel reverse search from \mts \cite{mts_tutorial}.



The most recent census of triangulations of cyclic polytopes that we are aware of is by Rambau and Reiner \cite[Table~1]{RambauReiner:Stasheff-Tamari}; we use their notation and introduce the parameter $c:=n-d$.
Note that there are two rather obvious typos in the rows $c\in\{10,11\}$ of the column $d=1$ in \cite[Table~1]{RambauReiner:Stasheff-Tamari}.
Apart from that we can confirm their results; cf.\ Table~\ref{tab:summary}.
Our new results are the values for $(c,d)\in\{(12,3),(8,5),(6,8),(5,14)\}$.

Our experiments used \mptopcom, version 1.0, on a cluster with four nodes, each of which comes with 2 x 8-Core Xeon E5-2630v3 (2.4 GHz) and 64GB per node. We ran \mptopcom with 40 threads. The operating system is SMP Linux 4.4.121.
For instance, the computation for $c=5$ and $d=14$, i.e., $n=19$ took 71191 seconds, i.e., less than 20 hours.

Azaola and Santos \cite[p.~30]{AzaolaSantos:2002} implicitly raised the following question.
\begin{question}\label{qst:beta}
  Is there an absolute constant $\beta>1$ such that, for all $n\geq 7$:
  \begin{equation}\label{eq:beta}
    \frac{1}{\beta} \ \leq \ \frac{\#\{\mbox{triangulations of }C(n,n-5)\}}{\#\{\mbox{triangulations of }C(n,2)\}} \leq \beta \enspace ?
  \end{equation}
\end{question}
This relates the row $c=5$ with the column $d=2$; these are marked in Table~\ref{tab:summary}.
From \mptopcom's results we can derive the series (\ref{eq:beta}) for $n\in\{7,8,\dots,19\}$:
\[
\begin{array}{c}
1,\
1.045,\
0.832,\
0.862,\
0.674,\
0.750,\
0.609,\\
0.767,\
0.673,\
1.001,\
0.972,\
1.760,\
1.910.
\end{array}
\]
Note that the sequence in \cite[p.~30]{AzaolaSantos:2002} lists the reciprocals of the above; moreover, that sequence contains two more (trivial) values for $n\in\{5,6\}$, which we omit.


\bibliographystyle{amsplain}
\bibliography{stasheff-tamari}

\end{document}